\documentclass[11pt]{article}

\usepackage[
        a4paper,
		margin = 2.3cm
]{geometry}
\usepackage{textcomp}
\usepackage{graphicx}

\usepackage{amssymb,amsmath,amsthm,color,mathtools}
\usepackage{float}
\usepackage[font=small]{caption}
\usepackage[font=footnotesize]{subcaption}
\usepackage{tikz}
\usepackage[pdftex,colorlinks,backref=page,citecolor=blue,linkcolor=blue,bookmarks=false]{hyperref}
\usepackage{comment}
\usepackage[shortlabels]{enumitem}
\usepackage{cleveref}
\usepackage{cite} 
\usepackage{pifont}

\usepackage{setspace}
\usepackage{tcolorbox}
\usepackage{xcolor}
\usepackage{framed}

\usepackage{algorithm, algpseudocode}

\newtheorem{theorem}{Theorem}[section]
\newtheorem*{theorem*}{Theorem}
\newtheorem{lemma}[theorem]{Lemma}
\newtheorem*{lemma*}{Lemma}
\newtheorem{claim}[theorem]{Claim}
\newtheorem*{claim*}{Claim}
\newtheorem{proposition}[theorem]{Proposition}
\newtheorem*{proposition*}{Proposition}

\newtheorem{corollary}[theorem]{Corollary}
\newtheorem*{corollary*}{Corollary}

\theoremstyle{definition}

\newcommand{\xb}{\mathbf{x}}

\newcommand{\reals}{\mathbb{R}}

\newcommand{\prm}{^{\prime}}

\newcommand{\abs}[1]{\left|#1\right|}

\newcommand{\nats}{\mathbb{N}}

\newcommand{\expect}[1]{\mathbb{E}\left[#1\right]}
\newcommand{\bigexpect}[1]{\mathbb{E}\big[#1\big]}

\newcommand{\tendsto}{\rightarrow}
\newcommand{\prob}[1]{\mathbf{P}\!\left[#1\right]}

\newcommand{\e}{{\rm e}}

\newcommand{\calP}{\mathcal{P}}

\newcommand{\R}{\mathbb{R}}

\newcommand{\eps}{\varepsilon}

\newenvironment{cframed}[1][blue]
  {\begin{tcolorbox}[colframe=#1,colback=white]}
  {\end{tcolorbox}}

\newcommand{\remove}[1]{}

\newcommand{\he}{\hat{e}}
\newcommand{\hG}{\hat{G}}
\newcommand{\htau}{\hat{\tau}}

\def \rgp {\mathbf{G}}
\newcommand{\hamiltonicity}{\mathrm{HC}}

\def \connectivity {{{\rm{con}}}}

\title{Building graphs with high minimum degree on a budget}
\author{
	Kyriakos Katsamaktsis\thanks{
		Department of Mathematics,
		University College London,
		Gower Street, London WC1E~6BT, UK.
		Email: \texttt{kyriakos.katsamaktsis.21}@\texttt{ucl.ac.uk}.
        Research supported by the Engineering and Physical Sciences Research Council [grant number EP/W523835/1].
	}
	\and
	Shoham Letzter\thanks{
		Department of Mathematics, 
		University College London, 
		Gower Street, London WC1E~6BT, UK. 
		Email: \texttt{s.letzter}@\texttt{ucl.ac.uk}. 
		Research supported by the Royal Society.
    }
}

\date{\today}

\begin{document}

\maketitle
\begin{abstract}
\setlength{\parskip}{\medskipamount}
\setlength{\parindent}{0pt}
\noindent
   We consider the problem of constructing a graph of minimum degree \(k\ge 1\) in the following controlled random graph process, introduced recently by Frieze, Krivelevich and Michaeli. 
   Suppose the edges of the complete graph on \(n\) vertices are permuted uniformly at random. A player, Builder, sees the edges one by one, and must decide irrevocably upon seeing each edge whether to purchase it or not.

   Suppose Builder
   purchases an edge if and only if at least one endpoint has degree less than \(k\) in her graph.
   Frieze, Krivelevich and Michaeli observed that this strategy succeeds in building a graph of minimum degree at least \(k\) by \(\tau_k\), the hitting time for having minimum degree \(k\).
   They conjectured that any strategy using \(\eps n\) fewer edges, where \(\eps>0\) is any constant, fails with high probability.

    In this paper we disprove their conjecture. We show that for \(k\ge 2\) Builder has a strategy which purchases \(n/9\) fewer edges and succeeds with high probability in building a graph of minimum degree at least \(k\) by \(\tau_k\).
    For \(k=1\) we show that any strategy using \(\eps n\) fewer edges fails with probability bounded away from 0, and exhibit such a strategy that succeeds with probability bounded away from 0.
\end{abstract}

\section{Introduction}

The \emph{uniform random graph} \(\rgp_{n,m}\) is the random graph drawn uniformly out of all labelled \(n\)-vertex graphs with \(m\) edges.
The series of papers by Erd\H{o}s and R\'enyi during 1959-1966~\cite{erdHos1960evolution,erdHos1961strength,erdos59,erdos66} started an in-depth study of random graphs, where they explored how a typical \(\rgp_{n,m}\) varies as \(m\) grows from \(0\) to \(\binom{n}{2}\).
The study of random discrete structures and their applications has grown to one of the most active and fruitful areas of research in combinatorics, with connections to many other fields of study, including theoretical computer science, number theory, analysis and statistical physics. 

Let \(N= \binom{n}{2}\) and \(e_1,\hdots, e_N\) be a uniformly random permutation of the edges of the complete graph on \(V=[n]\).
Then we can identify \(\rgp_{n,m}\) with the first \(m\) edges in the permutation.
The \emph{random graph process} is the nested sequence \(\left( \rgp_{n,m} \right)_{m=0}^N\) of random graphs.
One of the most striking discoveries of Erd\H{o}s and R\'enyi was that several \emph{increasing} properties appear rather suddenly in the random graph process.
Recall that a \emph{graph property} \(\calP\)
is a collection of graphs with vertices in \(V\), and
is \emph{increasing} if it is closed under the addition of edges with both ends in \(V\).
For example, the property of being Hamiltonian and the property of being connected are increasing properties.
They showed that for several increasing properties \(\calP\), including being connected, there is an \(m_0 = m_0(n)\) such that, when \(m\) is much larger than \(m_0\) (in an appropriate sense), \emph{with high probability}\footnote{We say that a sequence of events \((A_n)_{n\in \nats}\) holds \emph{with high probability} if \(\lim_{n\tendsto \infty} \prob{A_n} \tendsto 1\).} \(\rgp_{n,m}\) satisfies \(\calP\), and when \(m\) is much smaller than \(m_0\), with high probability \(\rgp_{n,m}\) does not satisfy \(\calP\).
The \emph{hitting time} for an increasing property \(\calP\) is defined to be the random variable
\(\tau_\calP = \min \{m\in [N]: \text{ \(\rgp_{n,m}\) satisfies \(\calP\)}  \}\).
A central problem in the theory of random graphs asks to determine asymptotically the hitting time for various graph properties.

The question of determining \(\tau_\hamiltonicity\), the hitting time for Hamiltonicity, was first raised by Erd\H{o}s and R\'enyi in 1960~\cite{erdHos1960evolution}, and subsequently studied by several authors.
The breakthrough is due to P\'osa~\cite{posa76} and Korshunov~\cite{korshunov76}, who independently showed in 1976 that, for some constant $C$, when \(m \ge Cn\log n\), with high probability \(\rgp_{n,m}\) is Hamiltonian.
Their result was improved by various authors until Koml\'{o}s and Szemer\'{e}di in 1983~\cite{komlos-szemeredi83} and independently Korshunov in 1977~\cite{korshunov77} 
showed that with high probability it is asymptotically the same as \(\tau_{\ge 2}\), the hitting time for having minimum degree at least \(2\).
Erd\H{o}s and R\'enyi in 1961~\cite{erdHos1961strength} proved that with high probability \(\tau_{\ge_2} = (1+o(1))\frac{n\log n}{2}\),
so this seems to settle the question of `when' the random graph process becomes Hamiltonian.
However, the precise control on the number of edges allows us to ask more probing questions.
Koml\'{o}s and Szemer\'{e}di in 1983~\cite{komlos-szemeredi83} claimed and a year later Bollob{\'a}s proved~\cite{bollobas84} that, in fact, with high probability, \(\tau_\hamiltonicity = \tau_{\ge 2}\).
That is, with high probability, the very edge that  increases the minimum degree to 2
is the one that makes the random graph process Hamiltonian.
This type of result is the strongest possible one can hope for regarding the emergence of an increasing property in the random graph process.

Clearly a graph with \(n\) vertices needs only \(n\) edges to be Hamiltonian, far fewer than the approximately
\(\frac{n \log n}{2}\) at the hitting time.
Is there an algorithm that finds a Hamilton cycle in the random graph process by time \(\tau_\hamiltonicity\) by considering only \(o(n\log n)\) edges?
Motivated by this question, Frieze, Krivelevich and Michaeli~\cite{budget} introduced the following controlled random graph process.
Suppose there is a player, Builder, who sees the edges in the random permutation \(e_1,\hdots, e_N\) one by one, and must decide irrevocably upon seeing \(e_i\) whether to purchase it or not.
A \emph{\((t,b)\)-strategy}  is an online algorithm (deterministic or randomised) that Builder follows in the above model where she sees only the first \(t\) edges in the random permutation of \(E(K_n)\) and is allowed to purchase at most \(b\) of them.
Builder's objective is for her graph to satisfy  a given increasing graph property \(\calP\).
For example, there is a simple 
\((\tau_{\connectivity}, n-1)\)-strategy such that Builder's graph is connected at the hitting time for connectivity: Builder purchases an edge if and only if it decreases the number of connected components in her graph.
Regarding Hamiltonicity, Frieze, Krivelevich and Michaeli~\cite[Theorem 3]{budget} showed that there exists a \((\tau_\hamiltonicity,Cn)\)-strategy such that with high probability \(B_{\tau_\hamiltonicity}\) is Hamiltonian, where \(C\) is a large constant.
Anastos~\cite{anastos} observed that \(C\) cannot be improved to 
\(1+o(1)\). 
However, in the same paper he showed that if Builder is allowed to see an additional \(\eps \tau_\hamiltonicity\) edges after \(\tau_\hamiltonicity\), then she can purchase at most \((1+o(1))n\) edges and ensure that with high probability her graph is Hamiltonian.

The current paper and 
Theorem 1 in~\cite{budget} consider the problem of Builder constructing a graph of minimum degree at least a constant \(k\ge 1\) at the hitting time 
\(\tau_k\) for the property of having minimum degree at least \(k\).
Let \(B_i\) be
Builder's graph just before the edge \(e_{i+1}\) is revealed.
Suppose that Builder follows the obvious greedy strategy, i.e.\ that she purchases 
\(e_i\) if and only if at least one of its ends has degree at most \(k-1\) in \(B_{i-1}\).
This strategy clearly succeeds in building a graph with minimum degree at least \(k\) at time \(\tau_k\).
By following this strategy, Builder's graph is distributed according to the \emph{k-th nearest neighbour} random graph model.
This model was studied by Frieze and Cooper~\cite{knn-cooper-frieze}, who proved that with high probability the resulting graph \(O_k\) has \((o_k+o(1)) n\) edges, for some explicit constant \(o_k \in (k/2, 3k/4]\) (cf.\ \Cref{cor:ok-bound}).
Hence, this gives a
\((\tau_k, (o_k+o(1))n)\)-strategy for Builder that (always) succeeds in constructing a graph of minimum degree at least \(k\).

Conjecture 7 in~\cite{budget} asserts that for any constant \(\eps>0\),
if Builder follows any \((\tau_k, (o_k-\eps)n)\)-strategy, with high probability her graph fails to have minimum degree at least \(k\).
Since \(o_1=3/4\) 
and \(o_k \ge \frac{k}{2} + \frac{3}{8}\) for \(k\ge 2\)
(cf.\ \Cref{cor:ok-bound}),
the next two theorems  disprove this conjecture, for the cases \(k\ge 2\) and \(k=1\) respectively.
\begin{theorem}\label{thm:k-at-least-2}
Let \(k\ge 2\) be an integer and \(\delta>0\) be constant.
Builder has a 
\((\tau_k, (k/2 + 2^{-k} + \delta)n)\)-strategy  that with high probability yields a graph with minimum degree at least \(k\).
\end{theorem}

\begin{theorem} \label{thm:k-1:strategy}
    Let \(C>0\) be constant.
    Builder has a \(\left(\tau_1, (1/2 + C^{-1/2}/2)n \right)\)-strategy that with probability at least 
    \((1-o(1))\frac{\sqrt{C}}{\e^C}\) yields a graph with minimum degree at least 1.
\end{theorem}
However, as the next theorem shows, a weaker version of \cite[Conjecture 7]{budget} holds for \(k=1\).
Recall that \(o_1 = 3/4\).
\begin{theorem} \label{thm:k-1:any-fails}
Let \(\eps>0\) be constant.
Any \(\left(\tau_1, (3/4-\eps)n\right)\)-strategy of Builder fails, with probability at least \(\frac{\eps^3}{100}\), in building a graph with minimum degree at least 1.
\end{theorem}

\noindent \textbf{Notation.}
Throughout \(n\) is assumed to be sufficiently large, and this is the only assumption made in all instances of asymptotic notation. 
We set
\(V=[n]\) and \(N= \binom{n}{2}\).
We use
\(e_1,\hdots, e_N\) to denote a uniformly random permutation of the edges of the complete graph on \(V\), so \(e_i\) is the edge at the \(i\)-th step of the random graph process. 
We denote by \(\rgp_{n,i}\) the graph at the \(i\)-th step of the random graph process,  i.e.\  the graph on \(V\) with edges \(\{e_1,\hdots, e_i\}\). 
The minimum degree of a graph \(G\) is  \(\delta(G)\) and the maximum degree is \(\Delta(G)\).

\section{Preliminaries}

The next theorem, which estimates the number of degree \(d\) vertices in \(\rgp_{n,Cn}\), follows from a straightforward application of the second moment method, see \cite[Theorem 3.3]{frieze-karonksi}.
It is also a simple consequence of a more general result of Bollob{\'a}s~\cite{bollobas1982vertices}.
\begin{theorem}\label{thm:vts-small-deg}
Let \(d\) be a fixed positive integer and
\(C , \delta\) be positive constants.
Let \(X_d\) be the number of vertices of degree \(d\) in \(\rgp_{n,Cn}\), and let \(\mu_d = \frac{C^d \e^{-C}}{d!}\cdot n\).
Then, with high probability,
\(\abs{X_d - \mu_d} \le \delta \mu_d\).
\end{theorem}
Next we state two concentration inequalities: Chernoff's bound and McDiarmid's inequality.
\begin{theorem}[{Chernoff's bound, see~\cite[eq.\ (2.5, 2.6) and Theorem 2.8]{jlr}}] \label{thm:Chernoff} 
    Let \(X\) be the sum of mutually independent indicator random variables and write 
    \(\mu = \expect{X}\).
    Then for any \(0<t\le \mu\),
    \[
    \prob{\abs{X- \mu} \ge t} \le 2\e^{-\frac{t^2}{3\mu}}.
    \]
\end{theorem}
\begin{theorem}[McDiarmid's inequality] \label{thm:mcdiarmid}
	\renewcommand{\prob}[1]{\mathbf{P}\big[#1\big]}
	Let \(X_1,\hdots,X_m\) be independent random variables with \(X_i\) taking values in a set \(S_i\).
	Let \(f: \prod_{i\in [m]} S_i \rightarrow \R\) be a function such that
	for any \(\xb, \xb\prm \in \prod_{i\in [m]} S_i\) differing only at the \(k^{\text{th}}\) coordinate we have
	\[
		\abs{f(\xb) - f(\xb\prm)} \le c_k,
	\]
	for some \(c_k \in \R\).
	Then, for every $t > 0$,
	\[
		\prob{\,\abs{f(X_1,\hdots,X_m)-\expect{f(X_1,\hdots,X_m)}} > t\,}
		\le 2\exp\left( -\frac{2t^2}{\sum_{k=1}^m c_k^2} \right).
	\]
\end{theorem}

We will use at several points the following straightforward consequence of Chernoff's bound and the union bound.
\begin{proposition} \label{prop:max-deg-bound}
    With high probability, the maximum degree of \(\rgp_{n,m}\) with \(m=O(n)\) is at most \(10 \log n\).
\end{proposition}

For a modern treatment of the next two theorems see chapter 4 of \cite{frieze-karonksi}.
The first one determines when, with high probability, the random graph process starts having minimum degree at least \(k\).
\begin{theorem}[Erd\H{o}s and Renyi~\cite{erdHos1961strength}] \label{thm:min-deg-k}
    Let \(m=\frac{n}{2} \left(\log n + (k-1)\log \log n + f(n) \right)\).
    \[ 
		\prob{\delta(\rgp_{n,m}) \ge k} 
		= \begin{cases}
			1-o(1) &\text{ if \(f(n) \tendsto \infty\)}\\
			o(1) &\text{ if \(f(n) \tendsto -\infty\).}
		\end{cases}
	\]
\end{theorem}
The next one shows that, with high probability, once a random graph has minimum degree at least 1, it is also connected.
\begin{theorem}[Bollob{\'a}s and Thomason~\cite{bollobas-thomason}]\label{thm:con-min-deg}
    In the random graph process,
    with high probability
    the hitting time for connectivity is the same as for minimum degree \(1\).
\end{theorem}

For two real random variables \(X,Y\), we say \(X\) \emph{stochastically dominates} \(Y\) if for all \(t \in \reals\)
\(
\prob{X\ge t} \ge \prob{Y\ge t}
\).
The next lemma is standard, see e.g.\ Section 23.9 of~\cite{frieze-karonksi} for a proof.
\begin{lemma}\label{lemma:stoc-dom-sum}
    Let \(Y_1,\hdots, Y_n\) be arbitrary real random variables, and let \(X_1,\hdots, X_n\) be mutually independent real random variables.
    
    Suppose that for all \(i\in [n]\) and  \(a_1,\hdots, a_{i-1} \in \reals\),
     \(Y_i\) conditioned on \(Y_1 = a_1,\hdots, Y_{i-1} = a_{i-1}\)
     stochastically dominates
     \(X_i\).
    Then \(\sum_{i=1}^n Y_i\)
    stochastically dominates
     \(\sum_{i=1}^n X_i\).

    Suppose instead that for all \(i \in [n]\) and  \(a_1,\hdots, a_{i-1} \in \reals\),
    \(Y_i\) conditioned on \(Y_1=a_1,\hdots, Y_{i-1} = a_{i-1}\) is stochastically dominated by \(X_i\). Then \(\sum_{i=1}^n X_i\) stochastically dominates \(\sum_{i=1}^n Y_i\).
\end{lemma}

The next lemma gives a simple \((O(n), \frac{k+o(1)}{2}n)\)-strategy that produces, with high probability, a graph in which almost all vertices have degree at least \(k\).
It is due to Frieze, Krivelevich and Michaeli~\cite{budget}.
For completeness, we give a detailed proof.
We remark that here and throughout the paper we use the convention that \(B_i\) is
Builder's graph just before the edge \(e_{i+1}\) is revealed.
\begin{lemma}[Lemma 2.15,~\cite{budget}]\label{lemma:very-greedy}
    Let \(k\ge 1\) be an integer, \(\eps\in (0,1)\) a constant and set \(C=k\eps^{-2}\).
    Suppose that for all \(i\in [Cn]\), if both ends of \(e_i\) have degree less than \(k\) in \(B_{i-1}\), then Builder purchases $e_i$.
    Then, with high probability, \(B_{Cn}\) has at most \(\eps n\) vertices of degree at most \(k-1\).
\end{lemma}
\begin{proof}
    \newcommand{\Xbi}{\mathbf{X}_{<i}}
    \newcommand{\abi}{\mathbf{a}_{<i}}
    \renewcommand{\prob}{\mathbf{P}}
    Let
    \[
    U_i = \{v\in V : \deg_{B_i}(v) \le k-1\}
    \]
    be the set of vertices of degree at most \(k-1\) in Builder's graph right after she has seen \(e_i\) and decided whether to purchase it.
    Hence, if both ends of \(e_i\) are in $U_{i-1}$ then $e_i$ is purchased by Builder.
    Let \(X_i\) be the indicator random variable for the event that either \(\abs{U_{i-1}} \le \eps n\), or both ends of \(e_i\) are in \(U_{i-1}\).
	Let \(a_1,\hdots, a_{i-1} \in \{0,1\}\) and write \(\Xbi = \abi \) for the condition 
    \(X_1 = a_1, \hdots, X_{i-1} = a_{i-1}\).
    Then
    \begin{align*}
		&\prob \Big[X_i=1 \: \Big |\: \Xbi = \abi \Big]
		\\
		& \qquad = \prob\Big[ X_i=1 \: \Big|\:  \Xbi = \abi,\, \abs{U_{i-1}} > \eps n \Big]\: \cdot \: \prob\Big[\abs{U_{i-1}} > \eps n \: \Big|\:  \Xbi = \abi\Big]
		\\
		&\qquad \qquad +\ 
		\prob\Big[X_i=1 \: \big| \: \Xbi = \abi,\, \abs{U_i}\le   \eps n\Big] \: \cdot \: 
		\left(1- \prob\Big[\abs{U_{i-1}} > \eps n \: \Big| \: \Xbi = \abi\Big]\right)
		\\
		&\qquad \ge\, \frac{\binom{\eps n}{2}}{\binom{n}{2}}\  \prob\Big[\abs{U_{i-1}} > \eps n \: \big| \: \Xbi = \abi\Big]
		+ 
		1 \cdot \left(1- \prob\Big[\abs{U_{i-1}} > \eps n \: \big|\: \Xbi = \abi \Big]\right)
		\\
		&\qquad \ge\,  2\eps^2/3.
    \end{align*}

    In particular, if \(Y_1,\hdots, Y_{Cn}\) are independent indicator random variables with mean \(2\eps^2/3\), then \(X_i\) conditioned on \(\Xbi = \abi\) stochastically dominates \(Y_i\).
    Therefore, by \Cref{lemma:stoc-dom-sum}, 
    \(\sum_{i=1}^{Cn} X_i\)  stochastically dominates \(\sum_{i=1}^{Cn}Y_i\).
    Chernoff's bound implies that, with probability at least \(1-\e^{-\Omega(k n)}\),
    \(\sum_{i=1}^{Cn} Y_i\) is at least \(1/2 \cdot \eps^2 \cdot Cn = kn/2\), and so the same holds for $\sum_{i = 1}^{Cn}X_i$.
    Therefore, either for some \(i\in [Cn]\), \(\abs{U_i} \le \eps n\), or Builder bought at least \(kn/2\) edges with both ends of degree at most \(k-1\) in Builder's graph right before purchasing them.
    In the former case the conclusion of the lemma follows since \(\abs{U_i}\) is clearly decreasing in \(i\), and in the latter case an easy calculation shows that \(\abs{U_{Cn}} = 0\).
\end{proof}

Recall that \(O_k\) is the random graph model which consists of the first \(k\) edges incident to each vertex in the random graph process.
Moreover, recall that
Frieze and Cooper~\cite{knn-cooper-frieze} showed that there exists a constant \(o_k \) such that with high probability \(e(O_k) = (o_k +o(1))n\).
Corollary 3.2 in \cite{budget} provides estimates for \(o_k\), following~\cite{knn-cooper-frieze}. 
The following proposition essentially follows from the proof of Corollary 3.2 in \cite{budget}.
The latter is not stated in this way because the authors are interested in the regime where \(k\) is large. 

\begin{proposition}[Corollary 3.2, \cite{budget}] \label{prop:ok-bound}
    \begin{equation*}
        o_k = \frac{k}{2} + \frac{1}{4}\sum_{i=0}^{k-1} \binom{2i}{i} 2^{-2i}
    \end{equation*}
\end{proposition}
\begin{proof}[Sketch]
    Following the proof of Corollary 3.2 in \cite{budget},
    set 
    \begin{equation*} 
         f(k) = \sum_{i,j=0}^{k} \binom{i+j}{i} 2^{-i-j}.
    \end{equation*}
    Then~\cite[Corollary 3.2]{budget} proves that
    \(o_k = k - \frac{f(k-1)}{4}\).
    Since, as noted in~\cite[Corollary 3.2]{budget},
    \(f(k) = f(k-1) + 2 - \binom{2k}{k}2^{-2k}\) and \(f(0)=1\), we have
    \[
    f(k) = 1 + 2k - \sum_{i=1}^{k} \binom{2i}{i} 2^{-2i}.
    \]
    Substituting in \(o_k = k - \frac{f(k-1)}{4}\) yields the required expression (summing now the binomial coefficients from \(i=0\)).
\end{proof}
The previous proposition directly implies the following.
\begin{corollary} \label{cor:ok-bound}
    \(o_1 = 3/4\) and \(o_k \ge k/2 + 3/8\) for \(k\ge 2\).
\end{corollary}

The next lemma will be useful when Builder follows a certain strategy for the first $Cn$ steps of the random graph, and then we wish to estimate the degree of vertices in the graph spanned by the edges exposed between steps $Cn+1$ and $\tau_k$.

\begin{lemma} \label{lemma:repetition-coupling}
    Let \(k\ge 1\) be an integer,
    and \(m \in [ c_1 n \log n, c_2 n \log n]\), where \(c_1,c_2>0\) are constants.
    Let \(H\) be a graph on \(V\) with maximum degree at most \(10 \log n\).
    Suppose we draw uniformly at random and with  repetition edges from the complete graph on \(V\) until there are \(m\) distinct edges.
    Then, with high probability, for every \(v\in V\) such that at least one of the edges drawn is incident to \(v\) in \(H\),
    there are at least \(k\) distinct edges incident to \(v\) which are not in \(H\).
\end{lemma}
\begin{proof}
Let \(L_v\subseteq E(H)\) be the set of edges incident to \(v\in V\), so \(\abs{L_v} \le 10 \log n\).
Let \(t\) be the (random) number of edges we draw with repetition until we see \(m\) distinct edges.

We claim that with high probability \(t\in [m, m+ (\log n)^3]\).
Indeed, let  \(X_i, i\in [m]\), be the random variable counting the number of edges drawn until the \(i\)-th distinct edge, and set \(X_0 = 0\).
Then \(X_i - X_{i-1}\) is a geometric random variable with mean \(\frac{1}{1- (i-1)/\binom{n}{2}}\), and
\begin{align*}
    \expect{X_m}=
    \sum_{i=1}^m \expect{X_i - X_{i-1}}
    &=
        \sum_{i=1}^m \frac{1}{1-(i-1)/\binom{n}{2}}
    \\ 
    &=
    \sum_{i=1}^m \left( 1+ \frac{i-1}{\binom{n}{2}} + O\left( \frac{i^2}{n^4}\right) \right)
    \\
    &= m + \Theta\left(\frac{m^2}{n^2}\right) + O\left(\frac{m^3}{n^4}\right)
    \\
    &= m + \Theta((\log n)^2).
\end{align*}
Let \(Y\) be the number of repeated edges until we see \(m\) distinct edges. Then the above calculation implies \(\expect{Y} = O((\log n)^2)\), whence from Markov's inequality with high probability \(Y \le (\log n)^3\).
Therefore \(t\in [m, m+ (\log n)^3]\), with high probability.

To complete the proof of the lemma we now show that, when drawing \(t\) edges with repetition, for some $t \in [m, m + (\log n)^{3}]$, with high probability, for every \(v\) with at least one  edge drawn from \(L_v\), there are at least \(k\) (distinct) incident edges not in \(L_v\).
Let \(\he_1,\hdots, \he_t\) be the drawn edges  and 
let \(A_{v,j}\) be the event that 
\(\he_j \in L_v\).
Let \(B_{v,j}\) be the event that 
the number of distinct edges in \(\{\he_1,\dots, \he_{t}\}\setminus \he_j\)  incident to \(v\) and disjoint from \(L_v\) is at most \(k-1\).
The lemma follows from the next claim.
\begin{claim}
	For every $t \in [m, m + (\log n)^{3}]$, the probability that there exist 
	\(v\in V, j\in [t]\) such that \(A_{v,j}\cap B_{v,j}\) holds is at most $1/n$.
\end{claim}
\begin{proof}
Fix \(v\in V\), \(j\in [t]\).
Because \(\he_1,\hdots, \he_t\) are drawn with repetition, the events \(A_{v,j}, B_{v,j}\) are independent.
We have
\(\prob{A_{v,j}} \le \frac{10 \log n}{\binom{n}{2}}\).

For \(n\) sufficiently large,
\begin{align*}
    \prob{B_{v,j}}
    & \le 
    \sum_{\ell=0}^{k-1} \, 
    \binom{t}{\ell}\, 
    \left( \frac{n-1}{\binom{n}{2}} \right)^{\ell} \, 
    \left(1- \frac{n-1-10\log n-(k-1)}{\binom{n}{2}} \right) ^{t -\ell}
	\\[.2em]
    & \le \sum_{\ell=0}^{k-1}
    (c_2n\log n + (\log n)^{3})^{\ell}\
    2^\ell n^{-\ell}\
    \left(1- \frac{(2-o(1))}{n}\right)^{(c_1-o(1))n \log n}
	\\[.2em]
    & \le \sum_{\ell=0}^{k-1}
    \left(2c_2 n \log n \right)^{\ell}\
    2^\ell\: n^{-\ell}\;
    \exp\left(-c_1\log n \right)
	\\[.2em]
    & \le \sum_{\ell=0}^{k-1}
	(4c_2)^{\ell} (\log n)^{\ell} \: n^{-c_1}
	\\[.2em]
    & \le (\log n)^k\,  n^{-c_1} \\
    & \le n^{-c_1/2}.
\end{align*}
Hence,  the probability there exist \(v\in V, j \in [t]\) such that \(A_{v,j} \cap B_{v,j}\) holds is, by the union  bound over the choice of \(v, j\), at most
\(
n\cdot t \cdot \frac{10\log n}{\binom{n}{2}} \cdot n^{-c_1/2} \le 1/n.
\)
\end{proof}
This completes the proof of the lemma.
\end{proof}

\section{Proof of \Cref{thm:k-at-least-2}}
Given \(k\ge 2\) and \(\delta>0\),
fix \(0<\eps<1\) and \(C>0\) such that:
\(\eps < \frac{\delta }{2k}\) and  \(C = k \eps^{-2} \).

Before delving into the details, we outline Builder's strategy, which we give formally in~\Cref{algo-deg-k}.
For the first \(Cn\) edges Builder's strategy is as follows.
For \(i\in [Cn]\), Builder purchases \(e_i = x_i y_i\) if and only if either  \(\deg_{B_{i-1}}(x_i), \deg_{B_{i-1}}(y_i) < k\) or at least one of \(x_i, y_i\) has degree 0 in \(\rgp_{n,i-1}\).
We call edges of the former kind \emph{efficient} and edges of the latter kind (which are not of the former kind) \emph{inefficient}.
Builder's strategy for \(i\ge Cn +1\) is to purchase \(e_i\) if and only if at least one end of \(e_i\) has degree less than \(k\) in \(B_{i-1}\).

\begin{algorithm}[h]
\caption{Builder's strategy for a graph with minimum degree at least \(k\) by \(\tau_k\)} \label{algo-deg-k}
    \begin{algorithmic}[1]
      \State \textbf{Input:} Let \(e_1,\dots, e_N\) be a uniformly random permutation of \(E(K_n)\).
      \State \(B_0 \gets \emptyset\)
      \For{\(i=1,\dots, Cn\)}
        \If{both ends of \(e_i\) have degree less than $k$ in $B_{i-1}$ or least one of has degree $0$}
            \State \(B_i \gets B_{i-1} \cup e_i\)
        \Else 
            \State \(B_i \gets B_{i-1}\)
        \EndIf
      \EndFor

      \For{\(i=Cn+1,\dots, \tau_k\)}
        \If{at least one end of \(e_i\) has degree less than \(k\) in \(B_{i-1}\)}
            \State \(B_i \gets B_{i-1} \cup e_i\)
        \Else 
            \State \(B_i \gets B_{i-1}\)
        \EndIf
      \EndFor
    \State \textbf{return} \(B_{\tau_k}\)
    \end{algorithmic}
\end{algorithm}

For \(i\in [Cn]\) let
\[
     Z_{i} = \{ v\in V: \deg_{\rgp_{n,i}}(v) =0  \}
\]
and
\[
    Y_{i} = \{v\in V:  1 \le \deg_{B_{i}}(v) \le k-1\}.
\]
The next claim shows that, with high probability, after the first \(Cn\) edges are exposed almost all vertices in Builder's graph have degree \(k\), yet her graph has relatively few edges.
\begin{claim} \label{claim:not-too-many-edges}
    Reveal the first \(Cn\) edges in the random graph process and suppose Builder follows \Cref{algo-deg-k}. Then with high probability
    \begin{enumerate}
        \item \(e(B_{Cn}) \le (k/2 + 2^{-k} + \delta/2)n\) and
        \item \(\abs{Y_{Cn} \cup Z_{Cn}} \le \eps n \).
    \end{enumerate}
\end{claim}
The next claim implies that, with high probability, by time \(\tau_k\) Builder's graph has minimum degree at least \(k\).
\begin{claim} \label{claim:zero-vts-get-deg-k-last}
	With high probability, for every \(v\in Y_{Cn}\), there are at least \(k-1\)  edges incident to \(v\)
	among \(e_{Cn+1},\hdots, e_{\tau_k}\).
\end{claim}
    
Before proving the claims, we show how to use them to derive \Cref{thm:k-at-least-2}.
\begin{proof}[Proof of \Cref{thm:k-at-least-2}]
    First we argue that following \Cref{algo-deg-k}, with high probability, \(\deg_{B_{\tau_k}}(v) \ge k\) for all \(v\in V\).
    Lines 4 and 11 of \Cref{algo-deg-k} imply that \(\deg_{B_{i}}(v) = 0\) if and only if \(\deg_{\rgp_{n,i}}(v) = 0\).
    Hence, for every \(v\notin Y_{Cn} \cup Z_{Cn}\),  \(\deg_{B_{Cn}}(v) \ge k\).
    By definition of \(\tau_k\) and \(Z_{Cn}\), every vertex in \(Z_{Cn}\) will see at least \(k\) incident edges among \(\{e_{Cn+1}, \hdots, e_{\tau_k}\}\), since none of the first \(Cn\) edges are incident to \(Z_{Cn}\).
    By \Cref{claim:zero-vts-get-deg-k-last}, with high probability, every vertex in \(Y_{Cn}\) will see at least \(k-1\) incident edges among \(\{e_{Cn+1},\hdots, e_{\tau_k}\}\).
    Therefore, in the second for-loop of \Cref{algo-deg-k} Builder will purchase at least \(k\) incident edges for each \(v\in Z_{Cn}\), and at least \(k-\deg_{B_{Cn}}(v)\) incident edges to each \(v\in Y_{Cn}\).
    Hence, with high probability, \(\deg_{B_{\tau_k}}(v) \ge k\) for all \(v\in V\).

    It remains to show that, with high probability, 
    \(e(B_{\tau_k}) \le (k/2 + 2^{-k} + \delta )n\).
    By \Cref{claim:not-too-many-edges},
    with high probability \(\abs{Y_{Cn} \cup Z_{Cn}} \le \eps n\).
    Hence, the number of edges Builder purchases during the second for-loop of \Cref{algo-deg-k} is at most \(k\eps n\).
    Therefore, with high probability, the total number of edges bought is at most
    \[
    e(B_{Cn}) + k \eps n
    \le 
    (k/2 + 2^{-k} + \delta/2 + k \eps)n
    \le (k/2 + 2^{-k} + \delta) n,
    \]
    where we used the upper bound on \(e(B_{Cn})\) from \Cref{claim:not-too-many-edges}.
\end{proof}

We now prove the two claims.

\begin{proof}[Proof of \Cref{claim:not-too-many-edges}]
\newcommand{\rgpnm}{\rgp_{n,m}}
    \Cref{lemma:very-greedy} implies the second point of the lemma.
    Clearly, the number of efficient edges bought is at most \(kn/2\).
    It remains to show that, with high probability, the number of inefficient edges is  at most \((2^{-k} + \delta/2 )n\).

    For \(r,s \in [n], r\ge s,\) let \(\Phi(r,s)\) be the collection of edges in the \emph{whole} random graph process (i.e.\ all \(\binom{n}{2}\) edges) which increase the degree of one of their ends to \(r\) and the other end to \(s\), and set \(\phi(r,s) = \abs{\Phi(r,s)}\).
    The inefficient edges are a subset of \(\rgp_{n,Cn}\), and by \Cref{prop:max-deg-bound}, with high probability \(\Delta(\rgp_{n,Cn}) \le 10 \log n\).
    Therefore, with high probability, the number of inefficient edges is at most 
     \(
    \sum_{r=k+1}^{10\log n} \phi(r,1).
    \)

    We want to show, using McDiarmid's inequality, that \(\phi(r,s)\) is, with high probability, concentrated around its mean, for $s \le r \le 10 \log n$.
    However, because \(\expect{\phi(r,s)}\) is linear (as we shall soon see) and is determined by quadratically many variables, applying McDiarmid's inequality directly to $\phi(r,s)$ would not give concentration with high probability.
    To avoid this issue, 
    let \(m=n (\log n)^{10}/2\), let \(\Phi_m(r,s) = \Phi(r,s) \cap \rgp_{n,m}\) be the subset of the first \(m\) edges in the random graph process that are in \(\Phi(r,s)\) and write \(\phi_m(r,s) = \abs{\Phi_m(r,s)}\).

	Define
	\begin{equation*}
		\mu_{r,s} = \frac{1}{2^{\delta(r,s)}} \binom{r+s-1}{s-1}2^{-r-s+1},
	\end{equation*}
	where $\delta(r,s)$ is $1$ if $r = s$ and $0$ otherwise (Kronecker delta).
	We will show that, with high probability, $\phi(r,s) = (1 + o(1))\mu_{r,s} n$ for all $1 \le s \le r \le 10 \log n$.

	To do so, we first note that $\phi(r,s) = \phi_m(r,s)$, with sufficiently high probability.
    To obtain concentration,
    we will reveal the first \(m\) steps of the random graph process in two stages.
	First we sample the graph at step \(m\) of the random graph process, \(\rgp_{n,m}\), and show that with sufficiently high probability \(\rgp_{n,m}\) is \emph{almost-regular} i.e.\ \( \deg_{\rgp_{n,m}}(v) = (1 + O((\log n)^{-2})) (\log n)^{10}\) for all \(v\in V\).
    We will then calculate the expectation of \(\phi_m(r,s)\) conditioned on \(\rgp_{n,m}\) being any given graph \(G\), and show that this expectation is $(1 + o(1))\mu_{r,s}n$ for every almost-regular $G$.
    Next, conditioned on \(\rgp_{n,m}\) being a fixed graph \(G\), we sample the first \(m\) steps of the random graph process by selecting a uniformly  random permutation of \(E(G)\).
    Using McDiarmid's inequality, we will show that \(\phi_m(r,s)\) conditioned on \(\rgpnm = G\) is, with high probability, concentrated around its mean. Then the required result, i.e.\ the concentration of \(\phi(r,s)\) around its mean for all relevant $r,s$, will follow easily.
    
	Fix $r,s$ such that $1 \le s \le r \le 10 \log n$.
	First, note that $\phi(r,s) = \phi_m(r,s)$ whenever $\delta(\rgp_{n,m}) \ge 10 \log n$, so $\phi(r,s) = \phi_m(r,s)$ with probability $1 - \exp\left(-\Omega((\log n)^{2}\right)$.

	Now let $G$ be any graph on $V$ with $m$ edges.
    We claim that $\bigexpect{\phi_m(r,s)\: |\: \rgp_{n,m} = G }$ equals
    \begin{align} \label{eq:expect-phi}
		\frac{1}{2^{\delta(r,s)}}\sum_{u, v \in V(G):\,\, uv \in E(G)} \frac{1}{\deg_{G}(u) + \deg_{G}(v)-1} 
		\,
		\frac{\dbinom{\, \deg_{G}(u)-1 \,}{r-1} \, \dbinom{\, \deg_{G}(v)-1}{s-1}\, }
		{\dbinom{\, \deg_{G}(u)\: +\: \deg_{G}(v) - 2 \,}{r+s-2}}.
	\end{align}
	This was (essentially) already noted in \cite{knn-cooper-frieze} for \(r,s \in [k]\) (\(k\) constant), when we do not condition on \(\rgp_{n,m}\) being any fixed graph.
    To see why this is true, note that for \(r\neq s\), the event \(e=uv \in \Phi_m(r,s)\) is the union of the (disjoint) events that \(e\) gives degree \(r\) to \(u\) and degree \(s\) to \(v\), and vice-versa. For \(r=s\) of course we have only one such event.
    Thus it suffices to show that the probability, over a random permutation of \(E(G)\), that \(uv\) gives degree \(r\) to \(u\) and \(s\) to \(v\) is
	\[
		\frac{1}{\deg_{G}(u) + \deg_{G}(v)-1}
		\,
		\frac{\dbinom{\, \deg_{G}(u) - 1\,}{r-1} \, \dbinom{\, \deg_{G}(v)- 1}{s-1}\, }
		{\dbinom{\, \deg_{G}(u)\: +\: \deg_{G}(v) - 2 \,}{r+s-2}}
	\]
    and then \eqref{eq:expect-phi} follows by linearity of expectation.
    The event in question is determined by the first \(r+s -1\) edges incident to \(\{u,v\}\) and occurs if both \(uv\) is the edge at position \(r+s -1\) in the permutation of the \(\deg_{G}(u) + \deg_{G}(v)-1\) edges incident to either \(u\) or \(v\); and among the first \(r+s-2\) edges incident to either \(u\) or \(v\), exactly \(r-1\) are incident to \(u\) and \(s-1\) are incident to \(v\). 
    It readily follows that this event occurs with the claimed probability. Indeed, the denominator is the number of ways of choosing $r+s-1$ edges among the $\deg_G(u) + \deg_G(v) - 1$ edges touching $u$ or $v$, with a single distinguished edge (the one in position $r+s-1$), and the numerator is the number of ways of choosing this many edges so that the distinguished edge is $uv$ and the remaining $r+s-2$ edges consist of $r-1$ edges touching $u$ and $s-1$ edges touching $v$.

	Now suppose that $G$ is almost-regular, i.e.\ that $\deg_G(v) = (1 + O((\log n)^{-2}) (\log n)^{10}$ for every $v \in V$.
	Using this assumption, \eqref{eq:expect-phi} and that $\binom{a}{b} = \left(1 + O(\frac{b^2}{a})\right) \frac{a^b}{b!}$, we get that for all such $G$,
	\begin{equation} \label{eq:cond-expect-phi}
        \bigexpect{\phi_m(r,s)\: |\: \rgp_{n,m} = G } =
        (1+o(1))\, \frac{1}{2^{\delta(r,s)}}\,
    \binom{r+s-1}{s-1}\, 2^{-r-s+1}\, n.
    \end{equation}
    We now show that \(\phi_m(r,s)\) conditioned on \(\rgp_{n,m} = G\) is concentrated around its mean.
    For this we use McDiarmid's inequality,
    which is applicable in the random graph process, because we can generate a uniformly random permutation of \(E(G)\) in the following way\footnote{Simply permuting uniformly at random \(E(G)\) does not allow us to use McDiarmid's inequality because the position of edges are not independent from one another.}.
    Fix a continuous probability distribution \(\mathcal{D}\) and for each \(e\in E(G)\) let \(X_e \sim \mathcal{D}\) be an independent sample from \(\mathcal{D}\).
    Ordering the edges from smallest sampled value to largest yields a uniformly random permutation of $E(G)$.
    To apply the inequality for \(\phi_m(r,s)\),
    suppose we unilaterally change the value \(X_{uv}\) sampled for the edge \(uv\in \rgp_{n,m}\), and leave the values sampled for all other edges \(\{ X_e: e\in E(G)\setminus uv\}\) unchanged.
    We claim that then \(\phi_m(r,s)\) changes by at most 4.
	To see this,
	observe that the graph spanned by \(\Phi_m(r,s)\) has \(\Delta(\Phi_m(r,s)) \le 2\): each vertex can only be incident to edges whose position in the random permutation increase its degree to \(r\) or \(s\).
	Changing \(X_{uv}\) can affect only edges incident to \(u\) or \(v\), 
    hence changing \(\phi_m(r,s)\) by at most \(4\).
    Thus, McDiarmid's inequality (\Cref{thm:mcdiarmid}) implies that
	\begin{align*}
		\prob{\, \big|\phi_m(r,s) - \bigexpect{\phi_m(r,s)\: |\: \rgp_{n,m} = G} \big| \ge n^{2/3}  } 
		& \le 2\exp\left(-\frac{2n^{4/3}}{16m}\right) \\
		& = \exp\left(-\Omega\left((\log n)^{2}\right)\right).
	\end{align*}
	In particular, using \eqref{eq:cond-expect-phi}, with probability $1 - \exp\left(-\Omega\left((\log n)^{2}\right)\right)$, we have that $\phi_m(r,s)$ conditioned on $\rgp_{n,m} = G$ is $(1 + o(1))\mu_{r,s}n$.
    
	Therefore, since \(\rgpnm\) is almost-regular with probability $1 - \exp\left(-\Omega\left((\log n)^{2}\right)\right)$ (by Chernoff's bound, \Cref{thm:Chernoff}), because we have $\phi_m(r,s) = \phi(r,s)$ with probability $1 - \exp\left(-\Omega\left((\log n)^{2}\right)\right)$, and by a union bound over $s,r$ with $1 \le s \le r \le 10\log n$, we have that $\phi(r,s) = (1 + o(1))\mu_{r,s} n$, with high probability.
    
    In particular, with high probability, \(\phi(r,1) = (1 + o(1))2^{-r}\), for every $r$ with \(k+1 \le r \le 10 \log n\).
    Hence, with high probability, the number of inefficient edges bought is at most
    \[
    (1+o(1)) \left(\sum_{r=k+1}^{10 \log n} 2^{-r} n\right) \le  (2^{-k} + \delta/2) n,
    \]
    as required. 
\end{proof}

\begin{proof}[Proof of \Cref{claim:zero-vts-get-deg-k-last}]
	By \Cref{prop:max-deg-bound}, with high probability \(\Delta(\rgp_{n,Cn}) \le 10 \log n\). Condition on this event.

	Draw with repetition edges 
	\((\he_i)_{i\ge 1}\) until
	\(t = \frac{n}{2} (\log n + (k-2) \log \log n + \log \log \log n)\) distinct edges are drawn.
	Let \(\hG\) be the graph of these \(t\) distinct edges and
	let \(\hG_i = \{\he_1,\hdots, \he_i\}\).

	We couple \((\hG_i)_{i\ge 1}\) with \((\rgp_{n,j})_{j\ge Cn +1}\) in the following natural way. 
	Suppose \(\rgp_{n,j}\) is the current stage of the random graph process and \(\he_i\) is the edge we have just drawn.
	If \(\he_i \notin \rgp_{n,j}\), we set \(e_{j+1} := \he_i\) and so
	update \(\rgp_{n,j+1} := \rgp_{n,j} \cup \{ \he_i \} \). Otherwise, \(e_{j+1}\) remains unrevealed, we do not update \(\rgp_{n,j}\), and keep drawing edges with repetition until we draw a new one. Clearly every edge not in \(\rgp_{n,j}\) has the same probability of being added, so the coupling indeed induces the random graph process.

	By \Cref{lemma:repetition-coupling} and \Cref{thm:min-deg-k}, with high probability, every \(v\in Y_{Cn}\) has \(k-1\) distinct incident edges in \(\hG\setminus \rgp_{n,Cn}\): indeed, by \Cref{thm:min-deg-k} with high probability every vertex \(v\) has at least \(k-1\) incident edges in \(\hG\), and for each \(v\) either all incident edges are disjoint from \(\rgp_{n,Cn}\), or otherwise by \Cref{lemma:repetition-coupling} at least \(k-1\) are.
	The latter is applicable since \(t=\Theta(n \log n)\) and \(\Delta(\rgp_{n,Cn}) \le 10 \log n\).

	Finally, by \Cref{thm:min-deg-k}, with high probability,
	\(\tau_k > t + Cn\).
	Hence, indeed we have \(E(\hG\setminus \rgp_{n,Cn}) \subseteq \{ e_{Cn+1}, \hdots, e_{\tau_k}\}\), with high probability.
\end{proof}

\section{Proof of \Cref{thm:k-1:any-fails,thm:k-1:strategy}}
Set \(\eps = C^{-1/2}\).
We assume \(\eps<3/4\), since otherwise both theorems easily hold.
\Cref{thm:k-1:any-fails} trivially holds because then Builder's graph is empty. \Cref{thm:k-1:strategy} holds because then Builder can purchase \((o_1+1/8)n\) edges by emulating \(O_1\) (using \(o_1 = 3/4\)).

We will keep track of the following sets of vertices.
\begin{align*}
    X_i &= \{v\in V: \deg_{B_i}(v) \ge 1 \}\\
    Y_i &= \{v \in V: \deg_{B_i}(v) =0 \text{ and } \deg_{\rgp_{n,i}}(v) \ge 1  \} \\
    Z_i &= \{v\in V: \deg_{\rgp_{n,i}}(v) =0 \}.
\end{align*}
For both theorems we will use the following lemma, which we prove at the end of this section.
Notice the first part of the lemma makes no assumption whatsoever on Builder's strategy.

\begin{lemma} \label{lemma:success-prob-formula}
    Let \(m\in [Cn]\) and suppose Builder has gone through the first \(m\) steps of the random graph process.
    For every strategy Builder may follow for edges \(e_j, j\ge m+1\),  \(B_{\tau_1}\) has an isolated vertex with probability at least \(\frac{(1-o(1))\abs{Y_{m}}}{\abs{Y_{m}} + \abs{Z_{m}}}\).
    
    Suppose that for \(j\ge m+1\) Builder follows the strategy of purchasing \(e_j\) if and only if at least one end of \(e_j\) is isolated in \(B_{j-1}\).
    Then with probability \(\frac{(1-o(1))\abs{Z_{m}}}{\abs{Y_{m}} + \abs{Z_{m}}}\), \(B_{\tau_1}\) has no isolated vertices.
\end{lemma}

We now prove \Cref{thm:k-1:strategy} according to which Builder has a $(\tau_1, (1/2 + C^{-1/2}/2))$-strategy which succeeds in building a graph with minimum degree at least $1$, with probability at least $(1 - o(1))\frac{\sqrt{C}}{e^C}$.

\begin{proof}[Proof of \Cref{thm:k-1:strategy}]
    
Builder has the following strategy, given formally in~\Cref{algo-deg-1} (any variables inside the algorithm agree with the notation defined outside),
for
constructing a graph with minimum degree at least 1 by time \(\tau_1\).
For the first \(Cn\) edges, Builder purchases an edge if and only if both endpoints are isolated.
For the remaining edges, she purchases an edge if and only if at least one end is isolated.
\Cref{lemma:very-greedy} implies that, with high probability, \(\abs{Y_{Cn} \cup Z_{Cn}} \le \frac{n}{\sqrt{C}}\), and 
\Cref{thm:vts-small-deg} yields that with high probability \(\abs{Z_{Cn}} \ge (1-o(1)) \e^{-C} n\).
Condition on these two events.
Then from
\Cref{lemma:success-prob-formula},
\(B_{\tau_1}\) has no isolated vertices with probability at least
\(
(1 - o(1))\frac{\abs{Z_{Cn}}}{\abs{Y_{Cn}} + \abs{Z_{Cn}}}
\ge 
(1-o(1))\frac{\sqrt{C}}{\e^{C}}.
\)
It remains to show that \(e(B_{\tau_1}) \le (1/2 + C^{-1/2}/2)n\).
Inspecting \Cref{algo-deg-1} (or simply the definitions), it is not hard to see  that \(Y_i \cup Z_i \subseteq Y_{i-1} \cup Z_{i-1}\): only lines \ref{line:Y-1} and \ref{line:Y-2} add elements to \(Y_i\), and in both cases vertices move to \(Y_i\) from \(Z_{i-1}\).
Also, it is clear that \(Z_i \subseteq Z_{i-1}\), see lines \ref{line:Z-1}, \ref{line:Z-2}, \ref{line:Z-3} and \ref{line:Z-4} of \Cref{algo-deg-1}.
Therefore, during the second for-loop, i.e.\ for edges \(e_j, j\ge Cn+1\), 
Builder will purchase only edges incident to \(Y_{Cn} \cup Z_{Cn}\).
Moreover, for \(j\ge Cn+1\), \(\abs{Y_j \cup Z_j}\) strictly decreases by at least 1 for each purchased edge.
For \(j\le Cn+1\), clearly \(\abs{Y_j \cup Z_j}\) decreases by exactly 2 (lines \ref{line:YZ-1}, \ref{line:YZ-2} and \ref{line:Z-1} of \Cref{algo-deg-1}).
Therefore 
\begin{align*}
	e(B_{\tau_1}) 
	& \le e (B_{Cn}) + \abs{Y_{Cn} \cup Z_{Cn}} \\
	& = (n- \abs{Y_{Cn} \cup Z_{Cn}}) /2 + \abs{Y_{Cn} \cup Z_{Cn}}
	\le (1+C^{-1/2})n/2,
\end{align*}
as required.
\end{proof}

\begin{algorithm}[h]
\caption{Builder's \((\tau_1, (3/4-\eps)n)\)-strategy for a graph with minimum degree 1} \label{algo-deg-1}
    \begin{algorithmic}[1]
      \State \textbf{Input:} Let \(e_1,\dots, e_N\) be a uniformly random permutation of \(E(K_n)\).
      \State \(B_0 \gets \emptyset\) \Comment{\(B_i\) is Builder's graph at step \(i\)}
      \State \(X_0 \gets \emptyset\) \Comment{non-isolated vertices in \(B_i\) at step $i$}
      \State \(Y_0 \gets \emptyset\) \Comment{isolated vertices in \(B_i\), non-isolated in random graph process at step \(i\)}
      \State \(Z_0 \gets V\)
      \Comment{isolated vertices in graph process at step \(i\) (so also in \(B_i\))}
      \For{\(i=1,\dots, Cn\)}
        \If{\(e_i\subseteq Y_{i-1} \cup Z_{i-1}\) }
		\Comment{if both ends of \(e_i\) are isolated in \(B_{i-1}\)} \label{line:YZ-1}
            \State \(B_i \gets B_{i-1} \cup e_i\)
            \State \(X_i \gets X_{i-1} \cup e_i\)
			\State \(Y_i \gets Y_{i-1} \setminus e_i\) \label{line:YZ-2}
			\State \(Z_i \gets Z_{i-1} \setminus e_i\) \label{line:Z-1}
        \Else 
            \State \(B_i \gets B_{i-1}\)
            \State \(X_i \gets X_{i-1}\)
			\State \(Y_i \gets (Y_{i-1} \cup e_i) \setminus X_i\) \label{line:Y-1}
            \Comment{add only previously isolated ends of \(e_i\)}
			\State \(Z_i \gets Z_{i-1} \setminus e_i\) \label{line:Z-2}
        \EndIf
      \EndFor

      \For{\(i=Cn+1,\dots, \tau_1\)}
        \If{\(e_i \cap (Y_{i-1} \cup Z_{i-1}) \neq \emptyset \) }
        \Comment{if at least one end of \(e_i\) is isolated in \(B_{i-1}\)}
            \State \(B_i \gets B_{i-1} \cup e_i\)
            \Comment{same updates as in first for-loop}
            \State \(X_i \gets X_{i-1} \cup e_i\)
            \State \(Y_i \gets Y_{i-1} \setminus e_i\)
			\State \(Z_i \gets Z_{i-1} \setminus e_i\) \label{line:Z-3}
        \Else 
        \Comment{same updates as in first for-loop}
            \State \(B_i \gets B_{i-1}\)
            \State \(X_i \gets X_{i-1}\)
			\State \(Y_i \gets (Y_{i-1} \cup e_i) \setminus X_i\) \label{line:Y-2}
            \Comment{add only previously isolated ends of \(e_i\)}
			\State \(Z_i \gets Z_{i-1} \setminus e_i\) \label{line:Z-4}
        \EndIf
      \EndFor
    \State \textbf{return} \(B_{\tau_1}\)
    \end{algorithmic}
\end{algorithm}

Now we prove \Cref{thm:k-1:any-fails} which asserts that any $(3/4-\eps, \tau_1)$-strategy fails in building a graph with minimum degree at least $1$, with probability at least $\frac{\eps^3}{100}$.
\begin{proof}[Proof of \Cref{thm:k-1:any-fails}]
This is an immediate consequence of \Cref{lemma:success-prob-formula} and the next claim.
\begin{claim} \label{claim:close-to-o1}
    If \(\abs{Y_{i}}\le \delta n\) for all \(i\in [Cn]\), then by time \(Cn\) Builder has purchased  at least \((o_1 -6\delta C- \delta - 2\e^{-C} - o(1))n\) edges.
\end{claim}
Set \(\delta = \frac{\eps}{50C}\)
and suppose Builder follows a strategy such that \(e(B_{\tau_1}) \le (o_1-\eps)n\).
Then \(e(B_{\tau_1}) < (o_1 -7\delta C-\delta-2\e^{-C})n\) (recalling \(\eps = C^{-1/2}\)), so  the contrapositive of \Cref{claim:close-to-o1} 
implies that for some \(i\in [Cn]\), \(\abs{Y_i} \ge \delta n\).
Hence, by \Cref{lemma:success-prob-formula}, \(B_{\tau_1}\) has an isolated vertex with probability at least 
\( (1-o(1))\frac{\abs{Y_i}}{\abs{Y_i} + \abs{Z_i}} \ge  \frac{\delta}{2} = \frac{\eps}{100C} = \frac{\eps^3}{100}\).
\begin{proof}[Proof of \Cref{claim:close-to-o1}]
    For \(i\in [Cn]\), let \(I_i\) be the indicator random variable for the event that \(e_i\) is incident to \(Y_{i-1}\) and \(\abs{Y_{i-1}}\le \delta n\).
    Conditioned on any choice of values for \(I_1,\hdots, I_{i-1}\), the expected value of \(I_i\) is at most 
    \(
    \frac{\delta n \cdot (n-1)}{\binom{n}{2}} = 2 \delta.
    \)
    Hence, \Cref{lemma:stoc-dom-sum} implies that \(\sum_{i=1}^{Cn} I_i\) is stochastically dominated by a sum of \(Cn\) independent indicator random variables with mean \(2\delta\) each, for which we can apply Chernoff's bound.
    The latter gives that, with high probability,
    \(\sum_{i=1}^{Cn} I_i \le 3\delta C n\). 
	Condition on this being the case.

    Because the claim assumes that \(\abs{Y_i} \le \delta n\) for all \(i\in [Cn]\), \(I_i\) is 1 if and only if \(e_i\) is incident to \(Y_{i-1}\).
    Hence,
    \(\sum_{i=1}^{Cn} I_i\) is precisely the number of edges \(e_i\) incident to \(Y_{i-1}\). 
    
    Let
    \(X_{Cn}' = X_{Cn} \cap \left(\bigcup_{i=1}^{Cn-1} Y_i \right)\) be
    the set of vertices for which Builder did not purchase the first incident edge in the random graph process, but did purchase a later edge (by step $Cn$).
    Since a vertex may move from \(Y_{i-1}\) to \(X_i\) only if there is an edge incident to \(Y_{i-1}\),
    \(
    \abs{X_{Cn}'} \le 2 \sum_{i=1}^{Cn} I_i \le 6 \delta C n.
    \)
    Therefore, all but at most \(6 \delta C n\) vertices in \(X_{Cn}\) are such that Builder bought the first incident edge to them.
    Let \(E_1 \subseteq E(B_{Cn})\) be the collection of edges that Builder bought and were the first incident edge to one of their endpoints.

    We will compare Builder's graph with the \(O_1\) graph at this stage of the random graph process.
    Recall that \(O_1\) consists of the first edge incident to each vertex in the random graph process, and that $O_1$ with high probability has at least $(o_1 - o(1))n$ edges; condition on this being the case.
    Let \(O_1^{Cn}\) be the subgraph of \(O_1\) with edges \(E(O_1) \cap \{e_1,\hdots, e_{Cn}\}\), so in particular \(E(O_1^{Cn}) \supseteq E_1\).
    Notice that any other edge in \(E(O_1^{Cn})\) is the first edge in the random graph process incident to some vertex in \(X_{Cn}' \cup Y_{Cn}\), so 
    \(\abs{E(O_1^{Cn})\setminus E_1} \le \abs{X_{Cn}' \cup Y_{Cn}}\).
    Therefore,
    \begin{equation} \label{eq:eo1cn-ub}
        \abs{E(O_1^{Cn})} \le \abs{E_1} + \abs{X_{Cn}'} + \abs{Y_{Cn}}
    \le \abs{E_1} + 6\delta C n + \delta n.
    \end{equation}
    Similarly, \(\abs{E(O_1) \setminus E(O_1^{Cn})} \le \abs{Z_{Cn}}\): every edge in \(E(O_1) \setminus E(O_1^{Cn})\) is the first edge in the random graph process incident to a vertex in \(Z_{Cn}\).
    Hence, with high probability,
    \(\abs{E(O_1)} -\abs{E(O_1^{Cn})} \le \abs{Z_{Cn}} \le 2 \e^{-C} n,\)
    using \Cref{thm:vts-small-deg} for the last inequality.
    Combining this with \eqref{eq:eo1cn-ub} we deduce that
    \[
        \abs{E(O_1)} \le \abs{E_1} + 6\delta C n + \delta n + 2\e^{-C} n \le e(B_{Cn}) + 6 \delta Cn + \delta n + 2\e^{-C}n.
    \]
    Rearranging gives \(e(B_{Cn}) \ge (o_1 - 6 \delta C  -  \delta - 2\e^{-C} - o(1)) n\), as required.
\end{proof}
This completes the proof of \Cref{thm:k-1:any-fails}.
\end{proof}

All that remains is to prove \Cref{lemma:success-prob-formula}.

\begin{proof}[Proof of \Cref{lemma:success-prob-formula}]

    As in \Cref{claim:zero-vts-get-deg-k-last},
    draw with repetition edges \((\he_i)_{i\ge 1}\) from \(E(K_n)\)
    and let \((\hG_i)_{i\ge 1}\) be the corresponding graph process, i.e.\ \(\hG_i\) is the (simple) graph spanned by \(\he_1,\hdots, \he_i\).
    We couple \((\hG_i)_{i\ge 1}\) with \((\rgp_{n,j})_{j\ge m+1}\) in the obvious way; see the proof of \Cref{claim:zero-vts-get-deg-k-last} for the details.
    Let 
    \[
    \htau_1 = \min \{i \ge 1: \text{for every } v\in Y_m \cup Z_m,\, \deg_{\hG_i}(v) \ge 1 \}.
    \]
    \begin{claim} 
    \label{claim:single-isolated}
        With high probability, a unique endpoint of \(\he_{\htau_1}\) is isolated in \(\hG_{\htau_1-1}\).
    \end{claim}
    We prove \Cref{claim:single-isolated} in the end. Condition on its conclusion holding.
    Let \(v \in Y_{m} \cup Z_{m}\) be the endpoint of \(\he_{\htau_1}\) in \(Y_{m} \cup Z_{m}\) that is isolated in \(\hG_{\htau_1-1}\) and notice that, by symmetry, \(v\) is uniformly distributed in  \(Y_{m} \cup Z_{m}\).
    
    Suppose \(v \in Y_m\).
    Then every vertex is non-isolated in the random graph process before \(\he_{\htau_1}\) is drawn:
    vertices in \(X_m \cup Y_m\) are non-isolated in \(\rgp_{n,m}\), and every vertex in \(Z_m\) has an incident edge among \(\{\he_1,\hdots, \he_{\htau_1-1}\}\), so indeed
    \(\rgp_{n,\tau_1} \subseteq \rgp_{n,m} \cup \hG_{\htau_1-1}\).
    Moreover, \(v\) is isolated in \(B_{\tau_1}\): none of the edges 
    \(e_j, j\in [m+1,\tau_1]\), is incident to \(v\), so \(v\) is isolated in \(B_{\tau_1}\), no matter what strategy Builder follows for these steps, since she purchased no edge from \(e_1,\hdots, e_m\) that is incident to \(v\).
	This proves the first part of the lemma, as the probability that $\hat{e}_{\htau_1}$ has a unique isolated vertex in $\hG_{\htau_1-1}$, which is in $Y_m$, is at least $(1 - o(1))\frac{|Y_m|}{|Y_m|+|Z_m|}$. 

    To prove the second part of the lemma, suppose \(v \in Z_m\).
    Then \(\he_{\htau_1}\) is the first incident edge to \(v\) in the random graph process, 
    and since all other vertices in \(Z_m\) have an incident edge among \(\he_1,\hdots, \he_{\htau_1-1}\), we deduce
     \(e_{\tau_1}  = \he_{\htau_1}\).
    Notice that Builder's strategy of purchasing the first incident edge to each vertex ensures that \(\deg_{B_{\tau_1}}(u) \ge 1\) for all \(u\in Z_m\), since each edge incident to \(u\in Z_m\) is disjoint from \(\rgp_{n,m}\), and thus she can purchase it.
    Hence the lemma follows if, with high probability, for every \(u\in Y_{m}\), there is at least one incident edge in \(\hG_{\htau_1} \setminus \rgp_{n,m}\).
    This follows from the assumption that $v \in Z_{\htau_1}$ which implies that every vertex in $Y_m$ has an incident edge in $\hG_{\htau}$, from \Cref{lemma:repetition-coupling}, which is applicable since, with high probability, \(\Delta(\rgp_{n,m})  \le 10 \log n\) by \Cref{prop:max-deg-bound} and 
    \(e(\hG_{\htau_1}) = \Theta(n\log n)\), the latter being a direct consequence of \Cref{thm:min-deg-k}. 
    Indeed, if \(e(\hG_{\htau_1})\ge n \log n\), by \Cref{thm:min-deg-k}, with high probability \(\hG_{\htau_1}\) has no isolated vertices.
    If \(e(\hG_{\htau_1}) \le \frac{n\log n}{3}\), then \(\hG_{\htau_1} \cup \rgp_{n,m} \subseteq \rgp_{n,m'}\) where \(m' = \frac{n \log n}{3} + Cn\), and
    \Cref{thm:min-deg-k} yields that, with high probability, \(\hG_{\htau_1} \cup \rgp_{n,m}\) has isolated vertices, which  must be in \(Z_m\).
    
    \begin{proof}[Proof of \Cref{claim:single-isolated}]
    Let \(t = \frac{n \left(\log n - \log \log n \right)}{2}\).
    The expected number of isolated vertices in \(\hG_t\) among \(Y_m \cup Z_m\) is at most
    \begin{align*} 
       &{\quad } \abs{Y_{m}\cup Z_m} \cdot \left(1- \frac{n-1}{\binom{n}{2}}\right)^{\frac{n}{2} \left(\log n - \log \log n \right)}\\
    &\le
    n \ \exp\left(-\frac{2}{n}\, \frac{n}{2} \left( \log n - \log \log n\right)\right)
    \\
    &= \log n.
    \end{align*}
    Therefore, by Markov's inequality, with high probability the number of isolated vertices in \(\hG_t\) is at most \((\log n )^2\).
    We condition on this event.
    We condition also on the event that \(\htau_1 = O(n\log n)\), which holds with high probability: indeed, the argument in the paragraph immediately preceding the current claim yields that, with high probability, \(e(\hG_{\htau_1}) = \Theta(n \log n)\), and the calculations at the beginning of \Cref{lemma:repetition-coupling} show that with high probability only \(o(n\log n)\) edges among \(\Theta(n \log n)\) drawn with replacement are repeated.
    
    The claim will follow if 
    among the remaining \(\htau_1 - t\) edges  (drawn with replacement) none has both ends among \(Y_m \cup Z_m\) which are isolated in \(\hG_t\).
    As the following calculation shows,
    the expected number of such edges is \(o(1)\):
    \[
    (\htau_1 - t) \ \frac{\binom{(\log n)^2}{2}}{ \binom{n}{2}}
    \le O(n \log n) \cdot O\left( \frac{(\log n)^4}{n^2}\right)
    = o(1).
    \]
    Hence, from Markov's inequality, with high probability no such edges are drawn.
\end{proof}
    This completes the proof of \Cref{lemma:success-prob-formula}.
\end{proof}

\section{Concluding remarks}
In this paper we studied the problem of constructing a graph with minimum degree \(k\ge 1\) in the controlled random graph process introduced in~\cite{budget}.
\Cref{thm:k-at-least-2} disproves \cite[Conjecture 7]{budget} in a strong form for \(k\ge 2\). 
Determining the optimal budget that
yields a graph with minimum degree at least \(k\) by time \(\tau_k\), with at least a given probability,
remains an interesting open problem.
Similarly,
\Cref{thm:k-1:any-fails,thm:k-1:strategy} disprove \cite[Conjecture 7]{budget} for \(k=1\). It would be interesting to determine the exact dependence between the budget and the optimal success probability.

\bibliographystyle{amsplain}
\bibliography{main}

\end{document}